\documentclass[12pt]{amsart}

\usepackage{ljm-auth}
\newtheorem{theor}{Theorem}[section]
\newtheorem{pr}[theor]{Proposition}
\newtheorem{cor}[theor]{Corollary}

\newtheorem{de}[theor]{Definition}
\newtheorem{rem}[theor]{Remark}
\newtheorem{que}[theor]{Question}
\newtheorem{ex}[theor]{Examples}

\author{Ural Bekbaev}

\crauthor{Ural Bekbaev} 

\tit{On relation between rational and differential rational invariants of surfaces
with respect to the motion groups}
\shorttit{On relation between} 

\begin{document}
\small \begin{center}{\textit{ In the name of
Allah, the Beneficent, the Merciful.}}\end{center}

\maketit
\address{Department of Science in Engineering, Faculty of Engineering,
IIUM, 50728, KL, Malaysia $\&$ Institute for Mathematical Research (INSPEM)\\ \email: bekbaev@iium.edu.my}

\abstract{The description of invariants of surfaces with respect to the motion groups is reduced to the description of invariants of parameterized surfaces with respect to the motion groups. Existence of a commuting system of invariant partial differential operators (derivatives) and a finite system of invariants, such that any invariant of the surface is a function of these invariants and their invariant derivatives,  is shown. The offered method is applicable in more general  settings than the "Moving Frame Method" does in differential geometry.} \notes{0}{

}

\section{Introduction}

The fundamental role of the "Moving Frame Method" and the "Jet
spaces" approach in the differential geometry is well known  \cite{Ol1,Ol2}. Nowadays they are involved nearly in every research done on differential geometry, particularly on invariants of surfaces with respect to the motion groups \cite{Ol2, Ol3}. In the present paper we offer a pure algebraic approach to the invariants of the surfaces. Taking into account some early results \cite{B1, B2} now, roughly speaking, one can conclude that the classification problems of the surfaces can be reduced to the classification problems of parameterized surfaces. We show that the description of differential invariants of surfaces can be lessen to a description of algebraic (not differential) invariants of the motion groups.
For the curves case the more detailed results are given in \cite{B3, B4}.

Let $n,m$ be any fixed natural numbers, $H$ be any subgroup of the general linear group $GL(n,\mathbb{R})$, where $\mathbb{R}$ stands for the field of real numbers. Consider its identity representation in $\mathbb{R}^n$: $(\mathbf{c},h)\mapsto \mathbf{c}h$, where $\mathbf{c}=(c_1,c_2,...,c_n)\in \mathbb{R}^n$ is a row vector, $h\in H$. In the invariant theory the descriptions of the algebra of $H$-invariant polynomials $\mathbb{R}[\mathbf{x}]^H$, and the field of $H$-invariant rational functions $\mathbb{R}(\mathbf{x})^H$ are considered as one of the main problems, where $\mathbf{x}=(x_1,x_2,...,x_n)$ and $x_1,x_2,...,x_n$ are algebraic independent variables over $\mathbb{R}$. For many, in particular for classic, groups the corresponding descriptions are known \cite{We}, \cite{Kr}. At the same time in geometry the description of invariants of $m$-parametric surfaces (patches) $\mathbf{x}=(x_1 (t_1,...,t_m),...,x_n (t_1,...,t_m))$ in $\mathbb{R}^n$ with respect to the motion group $H$ is also one of the important problems. In this case consideration the algebra of $H$-invariant $\partial$-differential polynomials $\mathbb{R}\{\mathbf{x};\partial\}^H$, as an important object is not convenient as far as even for finite $H$ it, as a $\partial$-differential algebra over $\mathbb{R}$, may have no finite system of differential generators over $\mathbb{R}$, where $\partial$ stands for $(\frac{\partial}{\partial t_1}, . . .,\frac{\partial}{\partial t_m})$ \cite{Kh}. For the corresponding field of the $H$-invariant $\partial$-differential rational functions $\mathbb{R}\langle \mathbf{x};\partial\rangle^H$ there is no such problem. It has a finite system of differential generators over $\mathbb{R}$ and $\mathbb{R}(\mathbf{x})^H\subset \mathbb{R}\langle \mathbf{x};\partial\rangle^H$. Therefore by algebraic point of view for the parameterized surfaces (patches) this is an important object to describe.
But for the theory of surfaces one has to consider $\partial$-differential rational functions in $\mathbf{x}$ which are not only $H$-invariant but also are invariant with respect to change of parameters
$(t_1,...,t_m)$. Let $\mathbb{R}\langle \mathbf{x};\partial\rangle^{(GL^{\partial},H)}$ stand for the set of all such invariant $\partial$-differential rational functions in $\mathbf{x}$ over $\mathbb{R}$. In general this field is not invariant with respect to $\partial$ and therefore it can not be considered as a $\partial$-differential field. We will show that there exists a commuting system of differential operators $\delta_1,\delta_2,. . .,\delta_m$ of $\mathbb{R}\langle \mathbf{x};\partial\rangle$ such that $\mathbb{R}\langle \mathbf{x};\partial\rangle^{(GL^{\partial},H)}$ is a $\delta$-differential field, where $\delta=(\delta_1,\delta_2,. . .,\delta_m)$. It should be noted that even in particular cases of $H$ the existence of such a commuting system of differential operators of $\mathbb{R}\langle \mathbf{x};\partial\rangle^{(GL^{\partial},H)}$ is not stated(guaranteed) \cite{Ol2, Ol3}.

In \cite{B1,B2} we have shown that the description of the $\partial$-differential field $\mathbb{R}\langle \mathbf{x};\partial\rangle^H$ can be reduced to the description of $\mathbb{R}(\mathbf{x})^H$ type invariants. One of the main results of the present paper states that  the description of the $\delta$-differential field $\mathbb{R}\langle \mathbf{x};\partial\rangle^{(GL^{\partial},H)}$  can be reduced to the description of $\mathbb{R}\langle \mathbf{x};\partial\rangle^H$ type invariants (Theorem 3.2).

Our approach to invariants of surfaces is an algebraic one and it enables one to avoid secondary particulars in exploring invariants of surfaces and, thanks to the symmetric product of matrices, to represent many relations in matrix form which are useful in finding invariants. The Differential Geometry offers "The Moving Frame" method to describe invariants of surfaces \cite{Ol1}. The prolonged actions on the surface jet spaces are not represented in matrix form and therefore one has to work with entries. It makes finding of invariants of surfaces in higher dimensional vector spaces hardly possible. Our method works not only in higher dimensional cases but also in more general settings than the case of surfaces in geometry. We are emphasizing this method by the use of differential geometry as far as it has come out from  the corresponding geometry problems.

Some of the results of this paper have been highlighted in \cite{B2, B5}. Before we have considered the hyper surface case in \cite{B6} and finite group $H$ case in \cite{B7}.
The used in present paper the notions and results from Differential Algebra
can be found in \cite{Ko,Ka}.

The organization of the paper is as follows. In the next section we present an algebraic approach to the surface problems, introduce definitions needed, notations and consider some results which will be used in future. Section 3 is about the main results of this paper. Section 4 deals with the construction of the commuting system of invariant partial differential operators. In this section for two dimensional surfaces in $\mathbb{R}^3$ and $\mathbb{R}^4$ the examples of the commuting systems of invariant partial differential operators are presented as well.
\section{Preliminaries}
\subsection{The symmetric product.}
 Further we need a new product of matrices, denoted by $\odot$, which can be called the symmetric product of matrices. For the proofs of the main properties of this product one can see \cite{B8}.

 For a positive integer $n$ let $I_n$ stand for all $n$-tuples with nonnegative integer entries with the
 following linear order: $\beta=
 (\beta_1,\beta_2,...,\beta_n)<\alpha=(\alpha_1,\alpha_2,...,\alpha_n)$ if and only if $\vert \beta\vert < \vert \alpha\vert$ or
$\vert \beta\vert = \vert \alpha\vert$ and $\beta_1>\alpha_1$ or $\vert \beta\vert = \vert \alpha\vert$, $\beta_1=\alpha_1$ and $\beta_2>\alpha_2$, et cetera
, where $\vert \alpha\vert$ stands for
$\alpha_1+\alpha_2+...+\alpha_n $. We consider in $I_n$ component-wise addition and subtraction, for example, $\alpha +\beta=(\alpha_1+\beta_1,...,\alpha_n+\beta_n)$.  We write $\beta \ll \alpha$ if $\beta_i \leq \alpha_i$
for all $i=1,2,...,n$, $\left(\begin{array}{c}
  \alpha \\
  \beta \\
\end{array}\right)$ stands for $\frac{\alpha!}{\beta!
(\alpha -\beta)!}$, $\alpha !=\alpha_1!\alpha_2!...\alpha_n!$ We consider also the following direct (or tensor) product:
$I_n \times I_m \longrightarrow I_{n+m}$. If $\alpha\in I_n$, $\beta\in I_m$ then $\alpha \times \beta= (\alpha, \beta)\in I_{n+m}$

 Let $n$, $n'$ and $n''$ be any fixed nonnegative integers (In the case of $n=0$ it is assumed that $I_n=  \{0\}$).

 For any  nonnegative
integer numbers $p',p$ let  $M_{n',n}(p',p;F)=M(p',p;F)$ stand for
all $"p'\times p"$ size matrices $A=(A^{\alpha'}_{\alpha})_{\vert
\alpha' \vert=p', \vert \alpha\vert=p}$ ($\alpha'$ presents
row, $\alpha$ presents column and $\alpha'\in I_{n'},\alpha\in
I_{n}$). The ordinary size of a such matrix
is $\left(\begin{array}{c}
  p'+n'-1 \\
  n'-1 \\
\end{array}\right)\times\left(\begin{array}{c}
  p+n-1 \\
  n-1 \\
\end{array}\right)$. Over such kind matrices in addition to the ordinary sum and
product of matrices we consider the following "product" $\odot$ as well:

\begin{de}  If $A\in M(p',p;F)$ and $B\in M(q',q;F)$ then
$A\odot B=C\in M(p'+q', p+q;F)$  such that for any
$\vert\alpha\vert=p+q$, $\vert\alpha'\vert=p'+q'$, where
$\alpha\in I_n,\alpha'\in I_{n'}$,
\[C^{\alpha'}_{\alpha}=\sum_{\beta,\beta'}\left(\begin{array}{c}
  \alpha \\
  \beta \\
\end{array}\right)
    A^{\beta'}_{\beta}B^{\alpha'-\beta'}_{\alpha-\beta},\]
where the sum is taken over all $\beta'\in I_{n'},\beta\in I_{n}$, for which $\vert
\beta'\vert=p'$, $\vert \beta\vert=p$, $\beta'\ll \alpha'$ and
$\beta\ll \alpha$.\end{de}

\begin{pr} For the above defined product the following
are true.

1. $A\odot B=B\odot A$.

2. $(A+B)\odot C=A\odot C+ B\odot C$.

3. $(A\odot B)\odot C=A\odot (B\odot C)$.

4. $ (\lambda A)\odot B=\lambda (A\odot B)$ for any
$\lambda\in F$.

5. $A\odot B=0$ if and only
    if $A=0$ or $B=0$.

6. $(A\odot V)B=(AB)\odot V$, where $V$ stands for any column matrix.

7. $A(B\odot H)=(AB)\odot H$, where $H$ stands for any row matrix. \end{pr}

  In future $A^{\odot m}$ means the $m$-th power of matrix $A$ with
respect to the $\odot$ product.

\begin{pr} If $\mathbf{r}=(r_1,r_2,...,r_{n})\in M(0,1;F)$,
$\mathbf{c}=(c_1,c_2,...,c_{n'})\in M(1,0;F)$, $h\in M(1,k;F)$ then
\[(\mathbf{r}^{\odot m})^{0}_{\alpha}=m!\mathbf{r}^{\alpha}, \hspace{1cm} (\mathbf{c}^{\odot m})^{\alpha'}_0=\left(\begin{array}{c}
  m \\
  \alpha'\\
\end{array}\right)\mathbf{c}^{\alpha'},\hspace{1cm}  (\mathbf{r}h)^{\odot m}=\frac{\mathbf{r}^{\odot m}}{m!}h^{\odot m},\] where $\mathbf{r}^{\alpha}$ stands for
$r_1^{\alpha_1}r_2^{\alpha_2}...r_n^{\alpha_n}$. For any square matrix $h\in M_{n,n}(1,1;F)$ and natural $k$ the equality \[\det(\frac{h^{\odot k}}{k!})=\det(h)^{\left(\begin{array}{c}
  n+k-1 \\
  n \\
\end{array}\right)}\] holds. \end{pr}

The following result has been presented in \cite{B9}.

 \begin{theor}Let $\mathbf{A}$ be any associative algebra, with $1$, over a field of zero characteristic and let $\mathbf{C}$ stand for the center of $\mathbf{A}$. If $\partial =(\partial^1,\partial^2,...,\partial^{n'})$, $\delta =(\delta^1, \delta^2,..., \delta^n)$ are column vectors of commuting system of differential operators of the algebra $\mathbf{A}$ for which $\partial = g\delta$, where $g^i_j\in \mathbf{C}$, $\partial^kg^i_j=\partial^ig^k_j$ for all $i,k=1,2,...,n'$, $j=1,2,...,n$ then for any natural $m$ the following equality is true
$\frac{\partial^{\odot m}}{m!}=$
\[ \sum_{k=1}^m k!\sum_{ |\alpha |=k,  \|\alpha\|=m } \frac{(\frac{\partial^{\odot 0}}{0!}\odot g)^{\odot \alpha_0}}{\alpha_0!}\odot
\frac{(\frac{\partial^{\odot 1}}{1!}\odot g)^{\odot \alpha_1}}{\alpha_1!}\odot ... \odot \frac{(\frac{\partial^{\odot m-1}}{m!}\odot g)^{\odot \alpha_{m-1}}}{\alpha_{m-1}!}\frac{\delta^{\odot k}}{k!},\]
where $\frac{\partial^{\odot 0}}{0!}\odot g=1\odot g=g$, $\alpha =(\alpha_0, \alpha_1,. . ., \alpha_{m-1})$, $\|\alpha\|= \alpha_0+2\alpha_1+. . .+m\alpha_{m-1}.$ \end{theor}

\subsection{The surface invariants by an algebraic point of view.}
Now let us consider the invariants of surfaces by the following algebraic point of view. Let $n$, $m$  be natural numbers and $H$ be a subgroup of $ GL(n,\mathbb{R})$, $G= Dif(\mathbf{B})$ be the group of
diffeomorphisms of the open unit ball $\mathbf{B}\subset \mathbb{R}^{m}$,
$\mathbf{u}:\mathbf{B}\rightarrow \mathbb{R}^{n}$ be a surface, where $\mathbf{u}$ is considered to be
infinitely smooth.

A function $f^{\partial}(\mathbf{u}(\mathbf{t}))$ of $u(\mathbf{t})=(u_1(\mathbf{t}),. . ., u_{n}(\mathbf{t}))$
and its finite number of derivatives relative to $\partial_1=
\frac{\partial}{\partial t_1}, . .
.,\partial_m=\frac{\partial}{\partial t_m}$ is said to be
invariant(more exactly, $ (G,H)$- invariant) if the equality
\[f^{\partial}(\mathbf{u}(\mathbf{t}))= f^{\delta}(\mathbf{u}(\mathbf{s}(\mathbf{t}))h)\] is valid for
any $\mathbf{s}\in G$, $h\in H$ and $\mathbf{t}\in \mathbf{B}$,
where $\mathbf{u}(\mathbf{t})$ is the row vector with coordinates
$u_1(\mathbf{t}),. . ., u_{n}(\mathbf{t}))$, $\mathbf{ s}(\mathbf{t})= (s_1(\mathbf{t}),. . .,s_m(\mathbf{t}))$,
$\delta_i= \frac{\partial}{\partial s_i}$.

Let $\mathbf{t}$ run $\mathbf{B}$ and $F= C^{\infty}(\mathbf{B}, \mathbb{R})$ be the differential ring
of infinitely smooth functions relative to differential operators
$\partial_1= \frac{\partial}{\partial t_1}, . .
.,\partial_m=\frac{\partial}{\partial t_m}$. The constant ring of
this differential ring is $\mathbb{R}$ i.e.  \[\mathbb{R}= \{a\in F: \partial_i a=
0\quad \mbox{at}\quad i=1,2,...,m \}.\] Every infinitely
smooth surface $\mathbf{u}:\mathbf{B}\rightarrow \mathbb{R}^{n}$ can be considered as an
element of the differential module $(F^{n};\partial_1, \partial_2,. .
.,\partial_m)$, where $\partial_i= \frac{\partial}{\partial t_i}$
acts on elements of $F^{n}$ coordinate-wisely. If elements of this
module are considered as the row vectors the above transformations,
involved in definition of invariant function, look like $\mathbf{u}=(u_1,.
. .,u_{n})\mapsto \mathbf{u}h,\quad \partial \mapsto g^{-1}\partial$
as far as \[\frac{\partial}{\partial t_i}=
\sum_{j=1}^{m}\frac{\partial s_j(\mathbf{t})}{\partial
t_i}\frac{\partial}{\partial s_j(\mathbf{t})},\]  where $g$ is the matrix with
the elements  $g^i_j=\frac{\partial s_j(\mathbf{t})}{\partial t_i}$
 at $i,j=1,2,...,m$, $\partial$  is the column vector with the "coordinates" $\frac{\partial}{\partial
t_1},. . .,\frac{\partial}{\partial t_m}$ , $h\in H$, $\mathbf{s}\in
G$. Moreover $\partial_ig^j_k= \partial_jg^i_k\ \mbox{for} \
i,j,k =1,2,...,m.$
Therefore for any differential field $(F;\partial_1, \partial_2,.
. .,\partial_m)$, i.e. $F$ is a field and $\partial_1,
\partial_2,. . .,\partial_m$ is a given commuting with each other
system of differential operators on $F$, one can consider the
transformations
\[\mathbf{u}=(u_1,. . .,u_{n})\mapsto \mathbf{u}h,\qquad \partial
\mapsto g^{-1}\partial,\] where $\mathbf{u} \in F^n,$ $h\in H$ is a
given subgroup of $ GL(n,C)$, $g\in
GL^{\partial}(m,F)$, \[GL^{\partial}(m,F)= \{ g\in  GL(m,F):
\partial_ig^j_k= \partial_jg^i_k\ \mbox{for} \ i,j,k=1,2,...,m\},\]
$\partial$ stands for the column-vector with the "coordinates"
$\partial_1,. . .,\partial_m$, and $C$ is the constant field of
$(F;\partial )$ i.e.
\[C= \{a\in F: \partial_i a= 0\quad \mbox{at}\quad
i=1,2,...,m \}.\]

It should be noted that for any $g \in GL^{\partial}(m,F)$ the
differential operators $\delta_1,\delta_2,. . .,\delta_m$, where
$\delta= g^{-1}\partial$, also commute with each other. So for any
$g\in GL^{\partial}(m,F)$ one can consider the differential field
$(F,\delta)$, where $\delta= g^{-1}\partial$. This transformation
is an  analogue of gauge transformations for the differential
field $(F;\partial)$.

In general case, the set $GL^{\partial}(m,F)$ is not
a group with respect to the ordinary product of matrices as far as
it is not closed with respect to this product. But by the use of
it a natural groupoid can be constructed with the base
$\{g^{-1}\partial :  g\in GL^{\partial}(m,F)\}$ \cite{W}.

Let $x_1,. . .,x_{n}$ be differential algebraic
independent variables over $F$ and $\mathbf{x}$ stand for the row vector
with coordinates $x_1,. . .,x_{n}$. We use the following notations
: $C[\mathbf{x}]$ the ring of polynomials in $x_1,. . .,x_{n}$ (over
$C$), $C(\mathbf{x})$ the field of rational functions in $\mathbf{x}$, $C\{
\mathbf{x};\partial \}$ the ring of $\partial$-differential polynomial
functions in $\mathbf{x}$ and $C\langle \mathbf{x};\partial \rangle$ is the field of
$\partial$-differential rational functions in $\mathbf{x}$ over $C$.

\begin{de} An element $f^{\partial}\langle
\mathbf{x}\rangle \in C\langle \mathbf{x}; \partial\rangle$ is said to be
 $(GL^{\partial}(m,F),H)$- invariant $(GL^{\partial}(m,F)$- invariant; $H$- invariant)
 if the equality
\[f^{g{-1}\partial}\langle \mathbf{x}h\rangle = f^{\partial}\langle
\mathbf{x}\rangle \] \[(\mbox{respectively,} f^{g{-1}\partial}\langle \mathbf{x}\rangle =
f^{\partial}\langle \mathbf{x}\rangle; f^{\partial}\langle \mathbf{x}h\rangle =
f^{\partial}\langle \mathbf{x}\rangle )\] is valid for any $g\in
GL^{\partial}(m,F),\ h\in H$.\end{de}

Let $C\langle \mathbf{x};
\partial\rangle^{(GL^{\partial}(m,F),H)}$  ($C\langle \mathbf{x}; \partial\rangle^{GL^{\partial}(m,F)}$,
$C\langle \mathbf{x}; \partial \rangle^H$ ) stand for the set of all such
$(GL^{\partial}(m,F),H)$- invariant (respectively,
$GL^{\partial}(m,F)$- invariant, $H$- invariant) elements of $
C\langle \mathbf{x};\partial\rangle$.

\begin{pr} If the system of differential
operators $\partial_1,. . .,\partial_m$ is linear independent over
$F$ then the differential operators $\delta_1,. .
.,\delta_m$, where $\delta= g^{-1}\partial, g\in GL(m,F)$, commute
with each other if and only if $g\in GL^{\partial}(m,F)$.
\end{pr}

\begin{proof} Let $g\in GL(m,F)$,\ $\delta= g^{-1}\partial $. It is
clear that linear independence of $\partial_1,. . .,\partial_m$
implies linear independence of $\delta_1,. . .,\delta_m.$ For any
$i,j=1,2,...,m$ we have $\partial_j=
\sum_{k=1}^mg^j_k\delta_k,$
 \[ \partial_i\partial_j=
\sum_{k=1}^m(\partial_i(g^j_k)\delta_k+
g^j_k\partial_i\delta_k)= \sum_{k=1}^m\partial_i(g^j_k)\delta_k
+ \sum_{k=1}^m\sum_{s=1}^mg^j_kg^i_s\delta_s\delta_k.\] Therefore
due to $\partial_i\partial_j= \partial_j\partial_i $ one has
\begin{eqnarray}\begin{array}{c}
\sum_{k=1}^m\partial_i(g^j_k)\delta_k +
\sum_{k=1}^m\sum_{s=1}^mg^j_kg^i_s\delta_s\delta_k=\\
\sum_{k=1}^m\partial_j(g^i_k)\delta_k +
\sum_{k=1}^m\sum_{s=1}^mg^j_kg^i_s\delta_k\delta_s
\end{array}\end{eqnarray}

If $\delta_k\delta_s = \delta_s\delta_k$ for any $k,s=
1,2,...,m$ them due to $(1)$ one has
\[\sum_{k=1}^m\partial_i(g^j_k)\delta_k  =
\sum_{k=1}^m\partial_j(g^i_k)\delta_k \] i.e. $\partial_i(g^j_k)=
\partial_j(g^i_k)$ for any  $i,j,k= 1,2,...,m$ because of linear independence of
$\delta_1,. . .,\delta_m$. Thus in this case we get
 $  g\in GL^{\partial}(m,F)$.

Vice versa, if $\partial_i(g^j_k)=
\partial_j(g^i_k)$ for any  $i,j,k= 1,2,...,m$  then due to $(1)$ at any $a\in F$ one has
$\sum_{k=1}^m\sum_{s=1}^mg^j_kg^i_s\delta_s\delta_k a=
\sum_{k=1}^m\sum_{s=1}^mg^j_kg^i_s\delta_k\delta_s a$ for any
$i,j= 1,2,...,m$. These equalities can be written in the
following matrix form \[g(\delta_1\delta a,. . .,\delta_m\delta
a)g^t= g(\delta_1\delta a,. . .,\delta_m\delta a)^tg^t,\] where
$t$ means the transpose. Therefore $(\delta_1\delta a,. .
.,\delta_m\delta a)= (\delta_1\delta a,. . .,\delta_m\delta a)^t$
i.e.
 $\delta_k\delta_s a= \delta_s\delta_k a$ for any
$k,s= 1,2,...,m$.\end{proof}

Note that for $g\in GL^{\partial}(m,F)$ and $\delta=
g^{-1}\partial$ the following equality is valid
\begin{eqnarray}\begin{array}{c} GL^{\delta}(m,F)=
g^{-1}GL^{\partial}(m,F). \end{array}\end{eqnarray}

\begin{pr} Let  $ (F,\partial_1,
\partial_2, ..., \partial_m)$ be a differential field of characteristic zero. The following three conditions are equivalent.

a) The system of differential operators  $ \partial_1, \partial_2,
 ...,\partial_m $ is linear independent over $ F$.

b) There is no nonzero $ \partial $- differential
 polynomial over  $ F$ which vanishes at all values
 of indeterminates from  $ F$.

c) If $p^{\partial}\{x_{11}, x_{12},..., x_{1m},
 x_{21},..., x_{2m},..., x_{m1},..., x_{mm}\}=
p^{\partial}\{(x_{ij})_{i,j= 1,2,...,m}\}$ is a differential
polynomial over $F$  such that $p^{\partial}\{g\}= 0 $  at any
$ g=(g^i_j)_{i,j= 1,2,...,m}\in GL^{\partial}(m,F)$, then
  $ p^{\partial}\{(t_{ij})_{i,j= 1,2,...,m}\}= 0 $ is also valid,
for any indeterminates  $ (t_{ij})_{i,j= 1,2,...,m}$,  for
which $ \partial_k t_{ij}=
 \partial_i t_{kj}$  at  $ i,j,k= 1,2,...,m.$ \end{pr}

\begin{proof} The equivalence of the conditions \emph{a}) and \emph{b}) is shown in
\cite{Ko}.

The implication $c) \Rightarrow b)$ is clear. In fact, if there are nonzero $
\partial $- differential
 polynomials over  $ F$ which vanish at all values
 of indeterminates from  $ F$ we can take with
minimal number of variables. Let $f\{z_1, z_2,. . .,z_l\}$ be such
a polynomial. If $l> 1$ and $f\{z_1, a_2,. . .,a_l\} \neq 0$ for
some $a_2,. . .,a_l\in F$ then it contradicts minimality of $l$.
Because the nonzero polynomial in one variable $f\{z_1, a_2,. .
.,a_l\}$ will vanish at all values of $z_1$ from $F$. If $l> 1$
and $f\{z_1, a_2,. . .,a_l\}= 0$ for all $a_2,. . .,a_l\in F$ then
considering $f\{z_1, z_2,. . .,z_l\}$ as a $ \partial $-
differential polynomial in $z_1$ over $F\{z_2, z_3,. . .,z_l\}$
once again we get a contradiction. Indeed, in this case at
least one of the coefficients of this polynomial has to be nonzero
$ \partial $- differential polynomial in $z_2, z_3,. . .,z_l$ ( as
$f\{z_1, z_2,. . .,z_l\}$ is a nonzero polynomial) and vanish at
all values of $z_2, z_3,. . .,z_l$ from $F$. It contradicts
the minimality of $l$. Thus $l= 1$ and we can consider nonzero
polynomial $p^{\partial}\{(x_{ij})_{i,j= 1,2,...,m}\}=
f\{x_{11}\}$ which vanishes at any $ g= (g^i_j)_{i,j=
1,2,...,m}$  $\in GL^{\partial}(m,F)$. This contradicts
\emph{c}).

Let us prove now that \emph{b}) implies \emph{c}). Assume that $
p^{\partial}\{(t_{ij})_{i,j= 1,2,...,m}\}\neq 0 $ for some
polynomial $p^{\partial}\{(x_{ij})_{i,j= 1,2,...,m}\}$ and $
p^{\partial}\{g\}= 0 $ for any $g\in GL^{\partial}(m,F)$. Due
to the equalities $ \partial_{k}t_{ij}=
 \partial_{i}t_{kj}$, $i,j,k= 1,2,...,m $
the nonzero $ p^{\partial}\{(t_{ij})_{i,j= 1,2,...,m}\} $ can
be represented as a polynomial $P$ of the monomials $
 {\partial_k}^{n_{k,i}} {\partial_{k+1}}^{n_{k+1,i}} ...
{\partial_m}^{n_{m,i}} t_{ki}, $ where  $ n_{j,i}$ are
nonnegative integers, $i,k= 1,2,...,m $.  Let $t_1, t_2,...,
t_m $ be any differential indeterminates over $ F$.  The
 inequality $ 0\neq p^{\partial}\{(t_{ij})_{i,j=
 1,2,...,m}\} $ and substitution $ t_{ij}=
 \partial_it_j $ give us a nonzero differential polynomial $\det(\partial_it_j)_{i,j=
 1,2,...,m}P$
 in $t_1, t_2,..., t_m  $ the value of which at any
 $(a_1, a_2,..., a_m)  $ from $F^m$ is zero. This contradicts
 \emph{\emph{\emph{b}}}).\end{proof}

\begin{pr} If  $\mathbf{a}=(a_1, a_2,..., a_m)$ and
$\mathbf{b}=(b_1, b_2,..., b_m)$ are any two nonzero row vectors from $F^m$
then there is an extension $(F_1,\partial)$ of $(F,\partial)$
where the equation $\mathbf{a}T=\mathbf{b}$ has solution in $GL^{\partial}(m,F_1)$.\end{pr}

\begin{proof} Assume, for example, that $a_1\neq 0$ and $
\{t_{ij}\}_{i= 2,3,...,m, j= 1,2,...,m}$ are such
 differential indeterminates over $ F$ that $
 \partial_{k}t_{ij}= \partial_{i}t_{kj}$  for  $i,k=
 2,3,...,m, j= 1,2,...,m$.
Consider \[(a_1, a_2,\dots, a_m)\begin{pmatrix}y_1& y_2&\dots &y_m \\
t_{21}& t_{22}&\dots &t_{im} \\ \hdotsfor[2]{4} \\ t_{m1}&
t_{m2}&\dots &t_{mm}\end{pmatrix}=(b_1, b_2,\dots , b_m)\] as a system of
linear equations in $y_1, y_2,..., y_m$.

 It has solution
\[(y_1, y_2,..., y_m)=(t_{11}, t_{12},. . ., t_{1m})= \frac{1}{a_1}(\mathbf{b}- \sum_{i=2}^{m}a_i(t_{i1}, t_{i2},.
. ., t_{im})) \]and the determinant of the corresponding matrix
$T=(t_{ij})_{i,j=1,2,...,m}$ is equal to
\[\frac{1}{a_1}\det\begin{pmatrix}b_1& b_2&.&.&.&b_m\\ t_{21}&
t_{22}&.&.&.&t_{im}\\ .& .&.&.&.&.\\ t_{m1}&
t_{m2}&.&.&.&t_{mm} \end{pmatrix}\] which is not zero because of  $\mathbf{b}\neq 0$.
Furthermore, if one defines \\ $\partial_1(t_{k1}, t_{k2},. . .,
t_{km})$ as
\[\partial_1(t_{k1}, t_{k2},. . ., t_{km})= \partial_k (t_{11}, t_{12},. . ., t_{1m})= \partial_k(\frac{1}{a_1}(\mathbf{b}-\sum_{i=2}^{m}a_i(t_{i1}, t_{i2},.
. ., t_{im}))) \] then $T\in GL^{\partial}(m, F_1)$, where
$F_1=F\langle
\{t_{i,j}\}_{i=\overline{2,m},j=1,2,...,m};\partial \rangle$.
\end{proof}

Further let $\mathbf{e}$ stand for the row vector $(1,0,0,...,0)\in F^m$
and $T$ stand for the matrix $(t_{ij})_{i,j= 1,2,...,m}$,
where $\{t_{ij}\}_{i,j= 1,2,...,m}$- are differential
indeterminates with the basic relations $ \partial_{k}t_{ij}=
 \partial_{i}t_{kj}$  for all $i,j,k= 1,2,...,m.$

\begin{cor} If $ (F; {\partial}_1,
 {\partial}_2, ..., {\partial}_m) $- is a
 differential field of characteristic zero,
  $ \partial_1, \partial_2, ..., \partial_m $
 is linear independent over  $ F$ and \\ $
 p^{\partial}\{t_1, t_2,..., t_m, t_{m+1},...,t_{m+l}\} $ is an arbitrary
 nonzero differential polynomial over $ F$ then \[\quad a) \  p^{\partial}\{\mathbf{e}T,t_{m+1},..., t_{m+l} \} \neq 0,\]
 \[b)\ p^{T^{-1}\partial}\{\mathbf{e}T^{-1}, t_{m+1},..., t_{m+l} \} \neq 0 ,\qquad  c)\ p^{T^{-1}\partial}\{\mathbf{e}T,t_{m+1},..., t_{m+l}\} \neq 0.  \] \end{cor}

\begin{proof} The proof of inequality \emph{a}) is evident due to
Propositions 2.7 and 2.8.

 Let us prove \emph{b}). If one assumes that
$ p^{T^{-1}\partial}\{\mathbf{e}T^{-1}, t_{m+1},..., t_{m+l}\}=0$ then in
particular for $T_0=(\partial_i x_j)_{i,j=1,2,...,m}$ one has
\[ p^{T^{-1}_0\partial}\{\mathbf{e}T^{-1}_0, t_{m+1},..., t_{m+l}\}=0.\] It should remain be true if one substitutes
$T^{-1}\partial$ for $\partial$ into this equality. As far as $\delta_0=
T_0^{-1}\partial$ is invariant with respect to such substitution
and $T_0$ is transformed to $T^{-1}T_0$ so \[
p^{T^{-1}_0\partial}\{\mathbf{e}T^{-1}_0T, t_{m+1},..., t_{m+l}\}=0.\] But
for any $S_0\in GL^{\delta_0}(m, F\langle x_1,...,x_m;\partial
\rangle )$ the equation $T^{-1}_0T= S_0$ has a solution in
$GL^{\partial}(m, F\langle x_1,...,x_m;\partial \rangle )$, namely
$T= T_0S_0$. Therefore due to Proposition 2.7 for the matrix of
variables $S=(s_{ij})_{i,j=1,2,...,m}$ such that
$\delta_{0i} s_{jk}= \delta_{0j} s_{ik}$ for all
$i,j,k=1,2,...,n$ one has
\[p^{\delta_0}\{\mathbf{e}, t_{m+1},..., t_{m+l}\}=0.\] Due to Corollary 2.9, part \emph{a}), we have $p^{\delta_0}\{t_1,..., t_{m+l}\}=0$ that is $p^{\partial}\{t_1,..., t_{m+l}\}=0$, which is a contradiction. The proof of \emph{c}) can be given similarly.\end{proof}

\section{The main results}

Further $(F,\partial )$ stands for a characteristic zero $\partial$-differential field and it is assumed that $\partial_1,. . .,\partial_m$ is linear independent
over $F$, where $\partial
=(\partial_1,. . .,\partial_m)$.

We use the following obvious fact repeatedly: If $t_1,...,t_l$ is
a $\partial$-algebraic independent system of variables   over $F$,
$g\in GL^{\partial}(m, F)$ and $p^{\partial}\{t_1,...,t_l\}$ is a
$\partial$-polynomial over $F$ then the following three equalities
\[p^{\partial}\{t_1,...,t_l\}=0, \ p^{g^{-1}\partial}\{t_1,...,t_l\}=0,\ p^{T^{-1}\partial}\{t_1,...,t_l\}=0. \]
 are equivalent.

Let us assume that for a given subgroup $H$ of $
GL(n,C)$ we have such a nonsingular matrix
\[M^{\partial}\langle \mathbf{x}\rangle = (M^{\partial}_{ij}\langle
\mathbf{x}\rangle)_{i,j=1,2,...,m}\] , where
$M^{\partial}_{ij}\langle \mathbf{x}\rangle\in C\langle \mathbf{x};\partial
\rangle$ and
 $\partial_{k}M^{\partial}_{ij}= \partial_{i}M^{\partial}_{kj}$  for  $i,j,k= 1,2,...,m$, that  \begin{eqnarray}\begin{array}{c}
M^{g^{-1}{\partial}}\langle \mathbf{x}h\rangle =
g^{-1}M^{\partial}\langle \mathbf{x}\rangle
\end{array}\end{eqnarray}
for any $g\in GL^{\partial}(m,F)$ and $h\in H$.

It is clear that  $C\langle \mathbf{x};\partial\rangle^H$ is a
finitely generated $\partial$-differential field over $C$ as a
subfield of $C\langle \mathbf{x};\partial \rangle$
 and $C\langle \mathbf{x};\partial\rangle^{(GL^{\partial}(m,F),H)}$ is a differential field with respect to $\delta = M^{\partial}\langle \mathbf{x}\rangle^{-1}\partial$. One of the most important questions is the differential-algebraic transcendence degree of $C\langle \mathbf{x};\partial\rangle^{(GL^{\partial}(m,F),H)}$ as a such $\delta $-field over $C$.

\begin{theor} The following equality \[\delta-\mbox{tr.deg.}C\langle \mathbf{x};
\partial\rangle^{(GL^{\partial}(m,F),H)}/C= n-m \ \mbox{holds}. \]\end{theor}

\begin{proof} First of all let us show that the system
 $M^{\partial}_{11}\langle \mathbf{x}\rangle ,M^{\partial}_{12}\langle \mathbf{x}\rangle,...,M^{\partial}_{1m}\langle \mathbf{x}\rangle $ is  $\delta$-algebraic independent over $C\langle \mathbf{x}; \partial\rangle^{GL^{\partial}(m,F)}$. If $p^{\delta}\{t_1,...,t_m\}$ is such a $\delta$-polynomial over $C\langle \mathbf{x}; \partial\rangle^{GL^{\partial}(m,F)}$ for which \[p^{\delta}\{M^{\partial}_{11}\langle \mathbf{x}\rangle ,M^{\partial}_{12}\langle \mathbf{x}\rangle,...,M^{\partial}_{1m}\langle \mathbf{x}\rangle\}=0\] then this equality should remain be true if one substitutes $g^{-1}\partial$ for $\partial$ into it. Therefore, as far as all coefficients of $p^{\delta}\{t_1,...,t_m\}$, $\delta$ are invariant with respect to such substitutions and $M^{g^{-1}{\partial}}\langle \mathbf{x}\rangle = g^{-1}M^{\partial}\langle \mathbf{x}\rangle $ one has
$p^{\delta}\{eg^{-1}M^{\partial}\langle \mathbf{x}\rangle\}=0$ i.e.
$p^{\delta}\{eT^{-1}M^{\partial}\langle \mathbf{x}\rangle\}=0$. But for
any $S_0\in GL^{\delta}(m, F\langle \mathbf{x}; \partial\rangle)$ the
equation $M^{\partial}\langle \mathbf{x}\rangle^{-1}T=S_0$ has solution
in $GL^{\partial}(m, F\langle \mathbf{x}; \partial\rangle)$, namely
$T=M^{\partial}\langle \mathbf{x}\rangle S_0$. It implies that for the
matrix of variables $S= (s_{ij})_{i,j=1,2,...,m}$, for which
$\delta_i s_{jk}=\delta_j s_{ik}$ ,$i,j,k=1,2,...,m$, one has
$p^{\delta}\{\mathbf{e}S^{-1}\}=0$. Due to Corollary 2.9, part \emph{a}) one has
$p^{\delta}\{t_1,...,t_m\}=0$.

Now let $f^{\partial}_1\langle \mathbf{x}\rangle,..., f^{\partial}_l\langle
\mathbf{x}\rangle $ be any system of elements of $C\langle \mathbf{x};
\partial\rangle^{GL^{\partial}(m,F)}$. We show that the system
\[M^{\partial}_{11}\langle \mathbf{x}\rangle ,M^{\partial}_{12}\langle
\mathbf{x}\rangle,...,M^{\partial}_{1m}\langle \mathbf{x}\rangle,
f^{\partial}_1\langle \mathbf{x}\rangle,..., f^{\partial}_l\langle \mathbf{x}\rangle
\] is $\delta$-algebraic independent over $C$ if and only if it is
$\partial$-algebraic independent over $C$.

Indeed, if this system is $\delta$-algebraic independent over $C$
and $p^{\partial}\{t_1,...,t_{m+l}\}$ is any polynomial over $C$
for which \[p^{\partial}\{\mathbf{e}M^{\partial}\langle \mathbf{x}\rangle ,
f^{\partial}_1\langle \mathbf{x}\rangle,..., f^{\partial}_l\langle
\mathbf{x}\rangle\}=0\] then it will have to remain be true if one
substitutes $g^{-1}\partial$ for $\partial$ into it, where $g\in
GL^{\partial}(m,F)$. It implies that
\[p^{T^{-1}\partial}\{\mathbf{e}T^{-1}M^{\partial}\langle \mathbf{x}\rangle ,
f^{\partial}_1\langle \mathbf{x}\rangle,..., f^{\partial}_l\langle
\mathbf{x}\rangle\}=0\]  as far as $f^{\partial}_1\langle \mathbf{x}\rangle,...,
f^{\partial}_l\langle \mathbf{x}\rangle$ are invariant with respect to such
transformations. But  \[T^{-1}\partial = (M^{\partial}\langle
\mathbf{x}\rangle^{-1}T)^{-1}M^{\partial}\langle \mathbf{x}\rangle^{-1}\partial =
(M^{\partial}\langle \mathbf{x}\rangle^{-1}T)^{-1}\delta\] and for any
$S_0\in GL^{\delta}(m,F\langle \mathbf{x};\partial\rangle)$ the equation
$M^{\partial}\langle \mathbf{x}\rangle^{-1}T=S_0$ has a solution in
$GL^{\partial}(m, F\langle \mathbf{x}; \partial\rangle)$, namely
$T=M^{\partial}\langle \mathbf{x}\rangle S_0$. Therefore for the matrix
of variables $S= (s_{ij})_{i,j=1,2,...,m}$, for which
$\delta_i s_{jk}=\delta_j s_{ik}$ ,$i,j,k=1,2,...,m$, one has
\[p^{S^{-1}\delta}\{\mathbf{e}S, f^{\partial}_1\langle \mathbf{x}\rangle,..., f^{\partial}_l\langle \mathbf{x}\rangle\}=0.\]
Due to our assumption $f^{\partial}_1\langle \mathbf{x}\rangle,..., f^{\partial}_l\langle \mathbf{x}\rangle$ is $\delta$-algebraic independent system over $C$ and therefore according to Corollary 2.9, part \emph{b}),  $p^{\delta}\{t_1,...,t_{m+l}\}=0$ i.e. $p^{\partial}\{t_1,...,t_{m+l}\}=0.$ So the system
\[M^{\partial}_{11}\langle \mathbf{x}\rangle ,M^{\partial}_{12}\langle
\mathbf{x}\rangle,...,M^{\partial}_{1m}\langle \mathbf{x}\rangle,
f^{\partial}_1\langle \mathbf{x}\rangle,..., f^{\partial}_l\langle \mathbf{x}\rangle
\] has to be $\partial$-algebraic independent over $C$.

Vise versa, let $M^{\partial}_{11}\langle \mathbf{x}\rangle
,M^{\partial}_{12}\langle
\mathbf{x}\rangle,...,M^{\partial}_{1m}\langle \mathbf{x}\rangle,
f^{\partial}_1\langle \mathbf{x}\rangle,..., f^{\partial}_l\langle \mathbf{x}\rangle
$ be $\partial$-algebraic independent over $C$. In this case first
of all the system $f^{\partial}_1\langle \mathbf{x}\rangle,...,
f^{\partial}_l\langle \mathbf{x}\rangle $ is $\delta$-algebraic independent
over $C$. Indeed, $f^{\partial}_1\langle \mathbf{x}\rangle,...,
f^{\partial}_l\langle \mathbf{x}\rangle $ is $\partial$-algebraic
independent over  $C\langle \{M^{\partial}_{ij}\langle \mathbf{x}\rangle
\}_{i,j=1,2,...,m};\partial \rangle$ because of
$\partial_kM^{\partial}_{ij}\langle
\mathbf{x}\rangle=\partial_iM^{\partial}_{kj}\langle \mathbf{x}\rangle$ for
$i,j,k=1,2,...,m$ and $\partial$-algebraic independence
\[M^{\partial}_{11}\langle \mathbf{x}\rangle
,M^{\partial}_{12}\langle
\mathbf{x}\rangle,...,M^{\partial}_{1m}\langle \mathbf{x}\rangle,
f^{\partial}_1\langle \mathbf{x}\rangle,..., f^{\partial}_l\langle \mathbf{x}\rangle
\] over $C$. But every nonzero $\delta$-differential polynomial
$p^{\delta}\{t_1,...,t_l\}$ over $C\langle
\{M^{\partial}_{ij}\langle \mathbf{x}\rangle
\}_{i,j=1,2,...,m};\partial \rangle$ can be considered as a
nonzero $\partial$- polynomial over $C\langle
\{M^{\partial}_{ij}\langle \mathbf{x}\rangle
\}_{i,j=1,2,...,m};\partial \rangle$. Therefore the supposition \\
$p^{\delta}\{f^{\partial}_1\langle \mathbf{x}\rangle,...,
f^{\partial}_l\langle \mathbf{x}\rangle\}=0$ leads to a contradiction that
$f^{\partial}_1\langle \mathbf{x}\rangle,..., f^{\partial}_l\langle
\mathbf{x}\rangle $ is $\partial$-algebraic independent over $C\langle
\{M^{\partial}_{ij}\langle \mathbf{x}\rangle
\}_{i,j=1,2,...,m};\partial \rangle$.

Let us assume that  for some polynomial $p^{\delta}\{t_1,...,t_{m+l}\}$ over $C$ one has
\[p^{\delta}\{eM^{\partial}\langle \mathbf{x}\rangle
,f^{\partial}_1\langle \mathbf{x}\rangle,..., f^{\partial}_l\langle
\mathbf{x}\rangle \}=0.\] This equality should remain be true if one substitutes
$g^{-1}\partial$ for $\partial$ into it, where $g\in
GL^{\partial}(m,F)$, which leads to
$p^{\delta}\{\mathbf{e}T^{-1}M^{\partial}\langle \mathbf{x}\rangle
,f^{\partial}_1\langle \mathbf{x}\rangle,..., f^{\partial}_l\langle
\mathbf{x}\rangle \}=0$. But the equation $M^{\partial}\langle
\mathbf{x}\rangle^{-1}T=S_0$ has a solution in $GL^{\partial}(m, F\langle \mathbf{x};
\partial\rangle)$ for any $S_0\in GL^{\delta}(m,F\langle
\mathbf{x};\partial\rangle)$ and therefore
\[p^{\delta}\{\mathbf{e}S^{-1}, f^{\partial}_1\langle \mathbf{x}\rangle,..., f^{\partial}_l\langle \mathbf{x}\rangle\}=0.\]
Now take into consideration that $f^{\partial}_1\langle
\mathbf{x}\rangle,..., f^{\partial}_l\langle \mathbf{x}\rangle $ is
$\delta$-algebraic independent  over $C$ and Corollary 2.9, part \emph{a}) to
see that $p^{\delta}\{t_1,...,t_{m+l}\}=0$. So it is shown that
the system \[M^{\partial}_{11}\langle \mathbf{x}\rangle
,M^{\partial}_{12}\langle
\mathbf{x}\rangle,...,M^{\partial}_{1m}\langle \mathbf{x}\rangle,
f^{\partial}_1\langle \mathbf{x}\rangle,..., f^{\partial}_l\langle \mathbf{x}\rangle
\] is $\delta$- algebraic independent over $C$ if and only if it
is $\partial$-algebraic independent over $C$. In particular it
shows that the existence of $M^{\partial}\langle \mathbf{x}\rangle$ with
property $(3)$ implies that $m\leq n$ as far as we have shown the  $\delta$-
algebraic independence of the system $M^{\partial}_{11}\langle \mathbf{x}\rangle
,M^{\partial}_{12}\langle
\mathbf{x}\rangle,...,M^{\partial}_{1m}\langle \mathbf{x}\rangle$ over $C$.

It is evident that the system of components of the matrix
$M^{\partial}\langle \mathbf{x}\rangle^{-1}=(N^{\partial}_{ij}\langle
\mathbf{x}\rangle )_{i,j=1,2,...,m}$ generates $C\langle \mathbf{x};\partial
\rangle$ over $C\langle \mathbf{x}; \partial\rangle^{GL^{\partial}(m,F)}$
as a $\delta$-differential field and $\delta_k
N^{\partial}_{ij}\langle \mathbf{x}\rangle =\delta_i
N^{\partial}_{kj}\langle \mathbf{x}\rangle$ for all
$i,j,k=1,2,...,m $, which implies that
$\delta$-tr.deg.$C\langle \mathbf{x};\partial \rangle /C\langle
\mathbf{x};\partial\rangle^{GL^{\partial}(m,F)} \leq m$. But it already has
been established that $M^{\partial}_{11}\langle \mathbf{x}\rangle
,M^{\partial}_{12}\langle
\mathbf{x}\rangle,...,M^{\partial}_{1m}\langle \mathbf{x}\rangle$ is
$\delta$-algebraic independent over $C\langle
\mathbf{x};\partial\rangle^{GL^{\partial}(m,F)}$ and therefore in reality
\[\delta-\mbox{tr.deg.}C\langle \mathbf{x};\partial \rangle /C\langle
\mathbf{x};\partial\rangle^{GL^{\partial}(m,F)}= m.\]

As a $\partial$-differential field $C\langle \mathbf{x};\partial\rangle$
over $C$ is generated by the elements of $C\langle
\mathbf{x};\partial\rangle^{GL^{\partial}(m,F)}$, as far as $x_1,...,x_n$ are in $C\langle
\mathbf{x};\partial\rangle^{GL^{\partial}(m,F)}$
and \\ $\partial$-tr.deg.$C\langle \mathbf{x};\partial\rangle
/C=n.$ It means that one can find such a system \\
$f^{\partial}_1\langle \mathbf{x}\rangle,..., f^{\partial}_{n-m}\langle
\mathbf{x}\rangle $  of elements of $C\langle
\mathbf{x};\partial\rangle^{GL^{\partial}(m,F)}$ for which the system
\[M^{\partial}_{11}\langle \mathbf{x}\rangle ,M^{\partial}_{12}\langle \mathbf{x}\rangle,...,M^{\partial}_{1m}\langle \mathbf{x}\rangle , f^{\partial}_1\langle \mathbf{x}\rangle,..., f^{\partial}_{n-m}\langle \mathbf{x}\rangle \]
is $\partial$-algebraic independent over $C$. As it has been shown
that in this case it is $\delta$-algebraic independent over $C$ as
well. Therefore \[\delta-\mbox{tr.deg.}C\langle \mathbf{x};\partial\rangle /C =n\ \mbox{and}\
\delta-\mbox{tr.deg.}C\langle
\mathbf{x};\partial\rangle^{GL^{\partial}(m,F)}/C = n-m.\]

Now to prove Theorem 3.1 it is enough to show that
$x_1,...,x_n$ are $\delta$-algebraic over  $C\langle \mathbf{x};
\partial\rangle^{(GL^{\partial}(m,F),H)}.$ Note that $x_1,...,x_n$ are linear independent over $C$ and therefore due to \cite{Ko} there exist nontrivial $m$-tuples $\alpha^1,\alpha^2,...,\alpha^n$ with nonnegative integer entries such that
  \[\det [\delta^{\alpha^1}\mathbf{x};\delta^{\alpha^2}\mathbf{x};...,\delta^{\alpha^n}x]\neq 0.\]  So for any nonzero $\alpha \in W^n$ one can consider the
following differential equation in $y$:
\[\det [\delta^{\alpha^1}\mathbf{x};\delta^{\alpha^2}\mathbf{x};...,\delta^{\alpha^n}\mathbf{x}]^{-1}\det
[\delta^{\alpha^1}\overline{x},\delta^{\alpha^2}\overline{x},...,\delta^{\alpha^n}\overline{x}
,\delta^{\alpha}\overline{x}]=0,\]
where $\overline{x}=(x_1,x_2,...,x_n,y)$, $\delta^{\alpha}=\delta^{(\alpha_1,\alpha_2,...,\alpha_m)}$
stands for
$\delta_1^{\alpha_1}\delta_2^{\alpha_2}...\delta_m^{\alpha_m}.$
All coefficients of this differential equation belong to $C\langle
\mathbf{x}; \partial\rangle^{(GL^{\partial}(m,F),H)}$ and $y= x_i$ is a
solution for this linear differential equation at any
$i=\overline{1,n}$. It implies that indeed
$\delta$-tr.deg.$C\langle \mathbf{x};
\partial\rangle^{(GL^{\partial}(m,F),H)}/C= n-m$.\end{proof}

The following result states that one can obtain a system of
generators of $(C\langle \mathbf{x}\rangle^{(GL^{\partial}(m,F),H)}, \delta
)$ from the given system of generators of $(C\langle
\mathbf{x};\partial\rangle^{H},\partial )$.

\begin{theor}If $C\langle
\mathbf{x};\partial\rangle^{H}$ as a $\partial$-differential
field over $C$ is generated by a system
$(\varphi^{\partial}_{i}\langle \mathbf{x}\rangle)_{i=\overline{1,l}}$ then
$\delta$-differential field $C\langle \mathbf{x};
\partial\rangle^{(GL^{\partial}(m,F),H)}$ is generated over $C$
by the system $(\varphi^{\delta}_{i}\langle
\mathbf{x}\rangle)_{i=\overline{1,l}}$.\end{theor}

\begin{proof} Let an irreducible $\frac{P^{\partial}\{
\mathbf{x}\}}{Q^{\partial}\{ \mathbf{x}\}}\in C\langle \mathbf{x};\partial\rangle^{H}$ be
$GL^{\partial}(m,F)$ -invariant. It means that for any $g\in
GL^{\partial}(m,F)$ one has the equality
\[P^{g^{-1}\partial}\{ \mathbf{x}\}Q^{\partial}\{\mathbf{x}\} =P^{\partial}\{\mathbf{ x}\}Q^{g^{-1}\partial}\{ \mathbf{x}\}\]
Therefore $P^{g^{-1}\partial}\{ \mathbf{x}\} = P^{\partial}\{\mathbf{x}\} \chi^{\partial}\langle g \rangle$.
 The function $\chi^{\partial}\langle T \rangle$( a "character" of $GL^{\partial}(m,F)$)
  has the following property
\[\chi^{\partial}\langle g_1g_2 \rangle =
\chi^{\partial}\langle g_1 \rangle\chi^{g^{-1}_1\partial}\langle
g_2 \rangle,\] for any $g_1 \in GL^{\partial}(m,F)$ and $g_2 \in
GL^{g^{-1}_1\partial}(m,F)$. But due to $(2)$ one has $g_2=
g^{-1}_1g$ for some $g\in GL^{\partial}(m,F)$ therefore
\[\chi^{\partial}\langle g\rangle =\chi^{\partial}\langle g_1
\rangle\chi^{g^{-1}_1\partial}\langle g^{-1}_1g \rangle,\] for any $g_1,g \in GL^{\partial}(m,F)$.
It implies that
\[\chi^{\partial}\langle T\rangle =\chi^{\partial}\langle S \rangle\chi^{S^{-1}\partial}\langle S^{-1}T \rangle,\] for any $ T=(t_{ij})_{i,j=
1,2,...,m}, S=(s_{ij})_{i,j= 1,2,...,m}$,  for which $
\partial_{k}t_{ij}=
 \partial_{i}t_{kj}, \partial_{k}s_{ij}=
 \partial_{i}s_{kj}$  at  $i,j,k= 1,2,...,m $.
The last equality guarantees that the function
$\chi^{\partial}\langle T\rangle $ can not vanish. Therefore
$\frac{P^{\partial}\{\mathbf{x}\}}{Q^{\partial}\{ \mathbf{x}\}}=\frac{P^{\delta}\{
\mathbf{x}\}}{Q^{\delta}\{ \mathbf{x}\}}$. \end{proof}

Let us assume that \[C\langle \mathbf{x};\partial \rangle^H=
C\langle\varphi^{\partial}_1\langle \mathbf{x}\rangle,
\varphi^{\partial}_2\langle
\mathbf{x}\rangle,...,\varphi^{\partial}_{l}\langle \mathbf{x}\rangle;\partial
\rangle.\] As far as all components of the matrix
$M^{\partial}\langle \mathbf{x}\rangle $ belong to   $C\langle
\mathbf{x};\partial \rangle^H$ it can be represented in the form
$M^{\partial}\langle \mathbf{x}\rangle=$ \[\overline{M}^{\partial}\langle \varphi^{\partial}_1\langle \mathbf{x}\rangle, \varphi^{\partial}_2\langle \mathbf{x}\rangle,...,\varphi^{\partial}_{l}\langle \mathbf{x}\rangle\rangle =(\overline{M}^{\partial}_{ij}\langle \varphi^{\partial}_1\langle \mathbf{x}\rangle, \varphi^{\partial}_2\langle \mathbf{x}\rangle,...,\varphi^{\partial}_{l}\langle \mathbf{x}\rangle\rangle )_{i,j=1,2,...,m},\] where $\overline{M}^{\partial}_{ij}\langle t_1,t_2,....,t_{l}\rangle \in C\langle t_1,t_2,....,t_{l};\partial \rangle$.
Therefore due to $M^{\delta}\langle \mathbf{x}\rangle= E_m$ one has
\[\overline{M}^{\delta}\langle \varphi^{\delta}_1\langle
\mathbf{x}\rangle, \varphi^{\delta}_2\langle
\mathbf{x}\rangle,...,\varphi^{\delta}_{l}\langle \mathbf{x}\rangle\rangle =E_m. \]

\begin{rem} The equality \[\chi^{\partial}\langle g\rangle
=\chi^{\partial}\langle g_1 \rangle\chi^{g^{-1}_1\partial}\langle
g^{-1}_1g \rangle,\] for any $g_1,g \in GL^{\partial}(m,F)$
resembles the property of character of the group $GL(m,F)$.
Therefore $\chi^{\partial}\langle T\rangle$ for which the above
equality is valid can be considered as a character of the groupoid
$GL^{\partial}(m,F)$. Description all such characters is an
interesting problem. Of course, $\chi^{\partial}\langle g\rangle=
\det(g)^k$, where $k$ is any integer number, are examples of such
characters. Are they all possible differential rational characters
of $GL^{\partial}(m,F)$? \end{rem}

\begin{theor}Any $\delta$-differential polynomial
relation over $C$ of the system
 $\varphi^{\delta}_1\langle \mathbf{x}\rangle, \varphi^{\delta}_2\langle \mathbf{x}\rangle,...,\varphi^{\delta}_{l}\langle \mathbf{x}\rangle $ is a
consequence of $\partial$-differential polynomial relations of the
system $\varphi^{\partial}_1\langle \mathbf{x}\rangle,
\varphi^{\partial}_2\langle
\mathbf{x}\rangle,...,\varphi^{\partial}_{l}\langle \mathbf{x}\rangle $ over $C$ and
the relations $\overline{M}^{\delta}\langle
\varphi^{\delta}_1\langle \mathbf{x}\rangle, \varphi^{\delta}_2\langle
\mathbf{x}\rangle,...,\varphi^{\delta}_{l}\langle \mathbf{x}\rangle\rangle =E_m.$\end{theor}

\begin{proof} Let $N^{\delta}\{\varphi^{\delta}_1\langle
\mathbf{x}\rangle, \varphi^{\delta}_2\langle
\mathbf{x}\rangle,...,\varphi^{\delta}_{l}\langle \mathbf{x}\rangle \} = 0$, where
$N^{\partial}\{ t_1,t_2,...,t_{l} \} \in C\{
t_1,t_2,...,t_{l},\partial\}$.

If $N^{\partial}\{\varphi^{\partial}_1\langle \mathbf{x}\rangle,
\varphi^{\partial}_2\langle
\mathbf{x}\rangle,...,\varphi^{\partial}_{l}\langle \mathbf{x}\rangle \} = 0$ then
it means that the above relation ($N^{\delta}\{
t_1,t_2,...,t_{l} \}$) of the system $\varphi^{\delta}_1\langle
\mathbf{x}\rangle, \varphi^{\delta}_2\langle
\mathbf{x}\rangle,...,\varphi^{\delta}_{l}\langle \mathbf{x}\rangle $ is a
consequence of the relation ($N^{\partial}\{ t_1,t_2,...,t_{l}
\}$) of the system $\varphi^{\partial}_1\langle \mathbf{x}\rangle,
\varphi^{\partial}_2\langle
\mathbf{x}\rangle,...,\varphi^{\partial}_{l}\langle \mathbf{x}\rangle $ i.e. it is
obtained by substitution $\delta$ for $\partial$ in
$N^{\partial}\{ t_1,t_2,...,t_{l} \}$.

If $N^{\partial}\{\varphi^{\partial}_1\langle \mathbf{x}\rangle,
\varphi^{\partial}_2\langle
\mathbf{x}\rangle,...,\varphi^{\partial}_{l}\langle \mathbf{x}\rangle \} \neq 0$
then consider\\
$N^{T^{-1}\partial}\{\varphi^{T^{-1}\partial}_1\langle
\mathbf{x}\rangle, \varphi^{T^{-1}\partial}_2\langle
\mathbf{x}\rangle,...,\varphi^{T^{-1}\partial}_{l}\langle \mathbf{x}\rangle \}$ as a
$\partial$-differential rational function in  variables
$T=(t_{ij})_{i,j=1,2,...,m}$, where $\partial_kt_{ij}
=\partial_it_{kj}$ for any $i,j,k=1,2,...,m$ over $C\langle
\mathbf{x};\partial\rangle$. Let $\frac{a^{\partial}_x\{
T\}}{b^{\partial}_x\{ T\}}$ be its irreducible representation and
the leading coefficient (with respect to some linear order) of
$b^{\partial}_x\{T\}$ be one. We show that in this case all
coefficients of $a^{\partial}_x\{ T\}, b^{\partial}_x\{T\}$ belong
to $C\langle \mathbf{x};\partial\rangle^H$.

Indeed, first of all
$N^{T^{-1}\partial}\{\varphi^{T^{-1}\partial}_1\langle
\mathbf{x}\rangle, \varphi^{T^{-1}\partial}_2\langle
\mathbf{x}\rangle,...,\varphi^{T^{-1}\partial}_{l}\langle \mathbf{x}\rangle \}$, as
a differential rational function in $\mathbf{x}$, is $H$- invariant
function, as much as $\varphi^{\partial}_i\langle \mathbf{x}\rangle \in
C\langle \mathbf{x};\partial\rangle^H$. This $H$-invariantness implies that
\[a^{\partial}_x\{T\}b^{\partial}_{xh}\{ T\}= b^{\partial}_x\{T\}a^{\partial}_{xh}\{T\}\] for any $h\in H$. Therefore
$b^{\partial}_{xh}\{T\}= \chi^{\partial}\langle
\mathbf{x};h\rangle b^{\partial}_x\{T\}$. But comparision of the
leading terms of both sides implies that in reality
$\chi^{\partial}\langle \mathbf{x};h\rangle= 1$ which in its turn
implies that all coefficients of $b^{\partial}_x\{T\}$ (as well as
$a^{\partial}_x\{T\}$) are $H$- invariant.

Therefore all coefficients of $a^{\partial}_x\{T\}$,
$b^{\partial}_x\{T\}$ can be considered as
${\partial}$-differential rational functions in
$\varphi^{\partial}_1\langle \mathbf{x}\rangle, \varphi^{\partial}_2\langle
\mathbf{x}\rangle,...,\varphi^{\partial}_{l}\langle \mathbf{x}\rangle$ and let $b^{\partial}_x\{T\}=
\overline{b}_{\varphi^{\partial}_1\langle \mathbf{x}\rangle,
\varphi^{\partial}_2\langle
\mathbf{x}\rangle,...,\varphi^{\partial}_{l}\langle \mathbf{x}\rangle}\{T\}$. Now
represent the numerator $a^{\partial}_x\{T\} $ as a
${\partial}$-differential polynomial function in $t_{ij}-
\overline{M}_{ij}^{\partial}\langle \varphi^{\partial}_1\langle
\mathbf{x}\rangle, \varphi^{\partial}_2\langle
\mathbf{x}\rangle,...,\varphi^{\partial}_{l}\langle \mathbf{x}\rangle\rangle $,where
$i,j=1,2,...,m$, for example, let
\[a^{\partial}_x\{T\}=\overline{a}^{\partial}_{\varphi^{\partial}_1\langle
\mathbf{x}\rangle, \varphi^{\partial}_2\langle
\mathbf{x}\rangle,...,\varphi^{\partial}_{l}\langle \mathbf{x}\rangle}\{T-
\overline{M}^{\partial}\langle \varphi^{\partial}_1\langle
\mathbf{x}\rangle, \varphi^{\partial}_2\langle
\mathbf{x}\rangle,...,\varphi^{\partial}_{l}\langle \mathbf{x}\rangle \rangle\}.\]
As such polynomial its constant term is zero because of\\
$N^{\delta}\{\varphi^{\delta}_1\langle \mathbf{x}\rangle,
\varphi^{\delta}_2\langle \mathbf{x}\rangle,...,\varphi^{\delta}_{l}\langle
\mathbf{x}\rangle \} = 0$. So
\[N^{T^{-1}{\partial}}\{\varphi^{T^{-1}{\partial}}_1\langle
\mathbf{x}\rangle, \varphi^{T^{-1}{\partial}}_2\langle
\mathbf{x}\rangle,...,\varphi^{T^{-1}{\partial}}_{l}\langle \mathbf{x}\rangle \}
=\]
\[\frac{\overline{ a}^{\partial}_{\varphi^{\partial}_1\langle
\mathbf{x}\rangle, \varphi^{\partial}_2\langle
\mathbf{x}\rangle,...,\varphi^{\partial}_{l}\langle
\mathbf{x}\rangle}\{T-\overline{M}^{\partial}\langle
\varphi^{\partial}_1\langle \mathbf{x}\rangle, \varphi^{\partial}_2\langle
\mathbf{x}\rangle,...,\varphi^{\partial}_{l}\langle \mathbf{x}\rangle \rangle
\}}{\overline{b}^{\partial}_{\varphi^{\partial}_1\langle \mathbf{x}\rangle,
\varphi^{\partial}_2\langle
\mathbf{x}\rangle,...,\varphi^{\partial}_{l}\langle \mathbf{x}\rangle}\{T\}}.\]

Substitution $T= E_m$ implies that
$N^{{\partial}}\{\varphi^{{\partial}}_1\langle \mathbf{x}\rangle, \varphi^{{\partial}}_2\langle \mathbf{x}\rangle,...,\varphi^{{\partial}}_{l}\langle \mathbf{x}\rangle \} =$
\[\frac{\overline{
a}^{\partial}_{\varphi^{\partial}_1\langle \mathbf{x}\rangle,
\varphi^{\partial}_2\langle
\mathbf{x}\rangle,...,\varphi^{\partial}_{l}\langle
\mathbf{x}\rangle}\{E_m-\overline{M}^{\partial}\langle
\varphi^{\partial}_1\langle \mathbf{x}\rangle, \varphi^{\partial}_2\langle
\mathbf{x}\rangle,...,\varphi^{\partial}_{l}\langle \mathbf{x}\rangle \rangle
\}}{\overline{b}^{\partial}_{\varphi^{\partial}_1\langle \mathbf{x}\rangle,
\varphi^{\partial}_2\langle
\mathbf{x}\rangle,...,\varphi^{\partial}_{l}\langle \mathbf{x}\rangle}\{E_m\}}.\]

Now consider the following $\delta$-differential rational function
over $C$:
\[\overline{N}^{\delta}\langle t_1,t_2,...,t_m \rangle=
N^{\delta}\{ t_1,t_2,...,t_{l}\}-\frac{\overline{
a}^{\delta}_{t_1,
t_2,...,t_{l}}\{E_m-\overline{M}^{\delta}\langle
t_1,t_2,...,t_{l} \rangle
\}}{\overline{b}^{\delta}_{t_1,t_2,...,t_{l}}\{E_m\}}.\] For this
function one has
$\overline{N}^{\delta}\langle \varphi^{\delta}_1\langle \mathbf{x}\rangle, \varphi^{\delta}_2\langle \mathbf{x}\rangle,...,\varphi^{\delta}_{l}\langle \mathbf{x}\rangle \rangle= 0$
as well as $ \quad \overline{N}^{\partial}\langle
\varphi^{\partial}_1\langle \mathbf{x}\rangle, \varphi^{\partial}_2\langle
\mathbf{x}\rangle,...,\varphi^{\partial}_{l}\langle \mathbf{x}\rangle\rangle = 0.$
The last equality(relation) means that it is a
consequence of relations of the system
$\varphi^{\partial}_1\langle \mathbf{x}\rangle, \varphi^{\partial}_2\langle
\mathbf{x}\rangle,...,\varphi^{\partial}_{l}\langle \mathbf{x}\rangle $. \end{proof}

\section{Construction of the systems of invariant partial differential operators}

In this section we discuss a construction of the matrix $M^{\partial}\langle \mathbf{x}\rangle$ with property $(3)$ for the given sub grouppoid $G$ of $GL^{\partial}(m,F)$ and subgroup $ H$ of $GL(n,C)$.

\begin{de} An element $f^{\partial}\langle
\mathbf{x}\rangle \in C\langle \mathbf{x}; \partial\rangle$ is said to be
 $(G,H)$- relative invariant
 if there exist integer numbers $k,l$ such that the equality
\[f^{g{-1}\partial}\langle \mathbf{x}h\rangle =\det(g)^k\det(h)^l f^{\partial}\langle
\mathbf{x}\rangle \] holds true for any $g\in
G, h\in H$.\end{de}

If $f^{\partial}\langle \mathbf{x}\rangle $ is such a nonzero relative invariant then by applying $\partial$ to the both sides of the equality \[f^{g{-1}\partial}\langle \mathbf{x}h\rangle =det(g)^kdet(h)^l f^{\partial}\langle \mathbf{x}\rangle, \] one gets \[g\delta\odot f^{\delta}\langle \mathbf{y}\rangle= det(h)^lkdet(g)^{k-1}(\partial\odot det(g)) f^{\partial}\langle \mathbf{x}\rangle+det(h)^ldet(g)^{k}(\partial\odot f^{\partial}\langle \mathbf{x}\rangle),\] where $\delta=g^{-1}\partial$, $\mathbf{y}=\mathbf{x}h$. Hence, assuming $k\neq0$, one has \[g\frac{\delta\odot f^{\delta}\langle \mathbf{y}\rangle}{kf^{\delta}\langle \mathbf{y}\rangle}= \frac{\partial\odot det(g)}{det(g)}+\frac{\partial\odot f^{\partial}\langle \mathbf{x}\rangle}{kf^{\partial}\langle \mathbf{x}\rangle}.\] If $\varphi^{\partial}\langle \mathbf{x}\rangle$ is another similar relative invariant with $k=k_1\neq 0, l=l_1$ then for the column vector \[X^{\partial}\langle \mathbf{x}\rangle=\frac{\partial\odot f^{\partial}\langle \mathbf{x}\rangle}{kf^{\partial}\langle \mathbf{x}\rangle}-\frac{\partial\odot \varphi^{\partial}\langle \mathbf{x}\rangle}{k_1\varphi^{\partial}\langle \mathbf{x}\rangle}\] one has
\[X^{\partial}\langle \mathbf{x}\rangle=gX^{\delta}\langle \mathbf{y}\rangle\]  for any $g\in G$.

If $k=0$ or $G\subset SL^{\partial}(m,F)$ then for $X^{\partial}\langle \mathbf{x}\rangle$  one can take $\frac{\partial\odot f^{\partial}\langle \mathbf{x}\rangle}{f^{\partial}\langle \mathbf{x}\rangle}$. One can try to construct the matrix $M^{\partial}\langle \mathbf{x}\rangle$ by the use of such column vectors.

Note that due to Theorem 2.4 one has the following matrix representations:
\[\left(
\begin{array}{c}
  1\\
  \partial\\
  \frac{\partial^{\odot 2}}{2!} \\
  \vdots\\
  \frac{\partial^\odot k}{k!}\\
\end{array}
\right)=\left(
\begin{array}{cccccc}
1&0&0&0&\cdots &0\\
0&  g & 0 & 0 &\cdots & 0\\
0&\frac{\partial\odot g}{2!}& \frac{g^{\odot 2}}{2!}&0&\cdots&0\\
\vdots&  \vdots&\vdots&\vdots&\cdots&0\\
0&\frac{\partial^{\odot k-1}\odot g}{k!}& *&*&\cdots& \frac{g^{\odot k}}{k!}\\
\end{array}
\right)
\left(
\begin{array}{c}
1\\
  \delta\\
  \frac{\delta^{\odot 2}}{2!} \\
  \vdots\\
  \frac{\delta^{\odot k}}{k!}\\
\end{array}
\right),\]
  \[\left(
\begin{array}{c}
    \partial\\
  \frac{\partial^{\odot 2}}{2!} \\
  \vdots\\
  \frac{\partial^\odot k}{k!}\\
\end{array}
\right)=\left(
\begin{array}{ccccc}
g & 0 & 0 &\cdots & 0\\
\frac{\partial\odot g}{2!}& \frac{g^{\odot 2}}{2!}&0&\cdots&0\\
\vdots&\vdots&\vdots&\cdots&0\\
\frac{\partial^{\odot k-1}\odot g}{k!}& *&*&\cdots& \frac{g^{\odot k}}{k!}\\
\end{array}
\right)
\left(
\begin{array}{c}
  \delta\\
  \frac{\delta^{\odot 2}}{2!} \\
  \vdots\\
  \frac{\delta^{\odot k}}{k!}\\
\end{array}
\right).\]

Let us consider the first matrix representation. The number of equations (rows) in the first matrix representation is equal to \[\left(\begin{array}{c}
  m-1+0 \\
  m-1 \\
\end{array}\right)+\left(\begin{array}{c}
  m-1+1 \\
  m-1 \\
\end{array}\right)+...+\left(\begin{array}{c}
  m-1+k \\
  m-1 \\
\end{array}\right)=\left(\begin{array}{c}
  m+k \\
  m \\
\end{array}\right).\]

The number of columns of $\mathbf{x}^{\odot l}$ is $\left(\begin{array}{c}
  n-1+l \\
  n-1 \\
\end{array}\right)$. Therefore whenever $1\leq l^k_{1}< l^k_{2}<...< l^k_{j_k}$ are such that
\[\left(\begin{array}{c}
  m+k \\
  m \\
\end{array}\right)=\sum_{i=1}^{j_k}\left(\begin{array}{c}
  n-1+l^k_i \\
  n-1 \\
\end{array}\right),\] which seems to happen for infinitely many $k$, due to Proposition 2.3 one has the matrix equality
\[\left(\left(
\begin{array}{c}
1\\
  \partial\\
  \frac{\partial^{\odot 2}}{2!} \\
  \vdots\\
  \frac{\partial^\odot k}{k!}\\
\end{array}
\right)\odot (\frac{\mathbf{x}^{\odot j_1}}{j_1!}, \frac{\mathbf{x}^{\odot j_2}}{j_2!},..., \frac{\mathbf{x}^{\odot j_k}}{j_k!})\right)Diag(\frac{h^{\odot j_1}}{j_1!},\frac{h^{\odot j_2}}{j_2!}, ...,\frac{h^{\odot j_k}}{j_k!})=\]
\[\left(
\begin{array}{cccccc}
1&0&0&0&\cdots &0\\
0&  g & 0 & 0 &\cdots & 0\\
0&\frac{\partial\odot g}{2!}& \frac{g^{\odot 2}}{2!}&0&\cdots&0\\
\vdots&  \vdots&\vdots&\vdots&\cdots&0\\
0&\frac{\partial^{\odot k-1}\odot g}{k!}& *&*&\cdots& \frac{g^{\odot k}}{k!}\\
\end{array}
\right)
\left(\left(
\begin{array}{c}
1\\
  \delta\\
  \frac{\delta^{\odot 2}}{2!}\\
  \vdots\\
  \frac{\delta^{\odot k}}{k!}\\
\end{array}
\right)\odot (\frac{\mathbf{y}^{\odot j_1}}{j_1!}, \frac{\mathbf{y}^{\odot j_2}}{j_2!},..., \frac{\mathbf{y}^{\odot j_k}}{j_k!})\right),\] where $\mathbf{y}=\mathbf{x}h$.

Taking the determinant of both sides of this equality results in
\[\det(h)^{\sum_{i=1}^{j_k}\left(
\begin{array}{c}
n+j_i-1\\
  n\\
  \end{array}
\right)}f^{\partial}_k\{\mathbf{x}\}=\det(g)^{\left(\begin{array}{c}
  m+k \\
  m+1 \\
\end{array}\right)}f^{\delta}_k\{\mathbf{y}\},\] where $f^{\partial}_k\{\mathbf{\mathbf{x}}\}$ stands for the $\det\left(
\begin{array}{c}
X\\
  \partial\odot X\\
  \frac{\partial^{\odot 2}}{2!}\odot X \\
  \vdots\\
  \frac{\partial^{\odot k}}{k!}\odot X\\
\end{array}\right)$, $X=(\frac{\mathbf{x}^{\odot j_1}}{j_1!}, \frac{\mathbf{x}^{\odot j_2}}{j_2!},..., \frac{\mathbf{x}^{\odot j_k}}{j_k!}),$ which means that $f^{\partial}_k\{\mathbf{x}\}$ is a relative invariant.

Let $f^{\partial}_{k_1}\{\mathbf{x}\}, f^{\partial}_{k_2}\{\mathbf{x}\},...,f^{\partial}_{k_{m+1}}\{\mathbf{x}\}$ be such relative invariants. By the use of them one can construct a matrix
$M^{\partial}\langle \mathbf{x}\rangle$ consisting of rows
\[M_i^{\partial}\langle \mathbf{x}\rangle=\frac{-\partial\odot f_{k_{i+1}}^{\partial}\langle \mathbf{x}\rangle}{\left(\begin{array}{c}
  m+k_{i+1} \\
  m+1 \\
\end{array}\right)f_{k_{i+1}}^{\partial}\langle \mathbf{x}\rangle}+\frac{\partial\odot f_{k_1}^{\partial}\langle \mathbf{x}\rangle}{\left(\begin{array}{c}
  m+k_1 \\
  m+1 \\
\end{array}\right)f_{k_1}^{\partial}\langle \mathbf{x}\rangle}.\] According to the construction the obtained matrix $M^{\partial}\langle \mathbf{x}\rangle$ is a nonsingular
matrix for which equality $(3)$ holds.

\begin{ex}
Now let us consider two dimensional surfaces in $\mathbb{R}^3$ and $\mathbb{R}^4$.

In $m=2$ case one has \begin{eqnarray}\left(\begin{array}{c}
  m+k \\
  m \\
\end{array}\right)_{k=1,2,3,...}=\{3, 6, 10, 15, 21, 22, 36, 45, 55, ...\}.\end{eqnarray}

A. For $n=3$ one has the sequence \begin{eqnarray}\left(\begin{array}{c}
  n-1+i \\
  n-1 \\
\end{array}\right)_{i=1,2,3,...}=\{3, 6, 10, 15, 21, 28, 36, 45, 55, ...\}\end{eqnarray}

Representing the elements of sequence $(4)$ as the sums of different elements of sequence $(5)$, when it is possible, one has:

1. $3=3$, which implies that $k=1, j^1_1=1$ and for $f^{\partial}_1\{x\}= \det\left(
\begin{array}{c}
\mathbf{x}\\
  \partial\odot \mathbf{x}\\
  \end{array}
\right)$ the equality  $f^{\delta}_1\{\mathbf{y}\}=\det(h)\det(g)^{-1}f^{\partial}_1\{\mathbf{x}\}$ holds.

2. $6=6$, which implies that $k=2, j^2_1=2$ and for \[f^{\partial}_2\{\mathbf{x}\}= \det\left(
\begin{array}{c}
\frac{\mathbf{x}^{\odot 2}}{2!}\\
\partial\odot \frac{\mathbf{x}^{\odot 2}}{2!} \\
 \frac{\partial^{\odot 2}}{2!}\odot\frac{\mathbf{x}^{\odot 2}}{2!}\\
  \end{array}
\right)\] the equality  $f^{\delta}_2\{\mathbf{y}\}=\det(h)^4\det(g)^{-4}f^{\partial}_2\{\mathbf{x}\}$ holds.

3. $10=10$, which implies that $k=3, j^3_1=3$ and for $f^{\partial}_3\{\mathbf{x}\}= \det\left(
\begin{array}{c}
\frac{\mathbf{x}^{\odot 3}}{3!}\\
  \partial\odot \frac{\mathbf{x}^{\odot 3}}{3!}\\
  \frac{\partial^{\odot 2}}{2!}\odot\frac{\mathbf{x}^{\odot 3}}{3!} \\
    \frac{\partial^\odot 3}{3!}\odot\frac{\mathbf{x}^{\odot 3}}{3!}\\ \end{array}
\right)$  the equality  $f^{\delta}_3\{\mathbf{y}\}=\det(h)^{10}\det(g)^{-10}f^{\partial}_3\{\mathbf{x}\}$ holds.

As we have noticed earlier one can use them to construct column vectors \[M_1^{\partial}\langle \mathbf{x}\rangle=\frac{\partial\odot f_2^{\partial}\langle \mathbf{x}\rangle}{-4f_2^{\partial}\langle \mathbf{x}\rangle}-\frac{\partial\odot f_1^{\partial}\langle \mathbf{x}\rangle}{-f_1^{\partial}\langle \mathbf{x}\rangle},\ \ \
M_2^{\partial}\langle \mathbf{x}\rangle=\frac{\partial\odot f_3^{\partial}\langle \mathbf{x}\rangle}{-10f_3^{\partial}\langle \mathbf{x}\rangle}-\frac{\partial\odot f_1^{\partial}\langle \mathbf{x}\rangle}{-f_1^{\partial}\langle \mathbf{x}\rangle},\] and for the second order square matrix $M^{\partial}\langle \mathbf{x}\rangle$ consisting of the columns $M_1^{\partial}\langle \mathbf{x}\rangle, M_2^{\partial}\langle \mathbf{x}\rangle$
the equality \[M^{g^{-1}{\partial}}\langle \mathbf{x}h\rangle =
g^{-1}M^{\partial}\langle \mathbf{x}\rangle,\] holds.

Note that in this particular case for $M^{\partial}\langle \mathbf{x}\rangle$ one can take the matrix $\left(\partial\odot\frac{f^{\partial}_2\{\mathbf{x}\}}{f^{\partial}_1\{\mathbf{x}\}^4},
\partial\odot\frac{f^{\partial}_3\{\mathbf{x}\}}{f^{\partial}_1\{\mathbf{x}\}^{10}}\right)$ as well.

 A'. If for the same purpose one uses the second matrix representation then for the number of equations (rows) in it one gets
\[\left(\begin{array}{c}
  m-1+1 \\
  m-1 \\
\end{array}\right)+...+\left(\begin{array}{c}
  m-1+k \\
  m-1 \\
\end{array}\right)=\left(\begin{array}{c}
  m+k \\
  m \\
\end{array}\right) -1\]
and
\begin{eqnarray} \left(\left(\begin{array}{c}
  m+k \\
  m \\
\end{array}\right)-1\right)_{k=1,2,3,...}=\{2, 5, 9, 14, 20, 27, 35, 44, 54, ...\}.\end{eqnarray}

Representing the elements of this sequence as the sums of different elements of sequence $(5)$, when it is possible, one has:

 1. $9=3+6$, which implies that $k=3, j^3_1=1, j^3_2=2$ and for $f^{\partial}_3\{x\}= \det\left(
\begin{array}{c}
  \partial\odot (\mathbf{x};\frac{\mathbf{x}^{\odot 2}}{2!})\\
  \frac{\partial^{\odot 2}}{2!}\odot (\mathbf{x};\frac{\mathbf{x}^{\odot 2}}{2!}) \\
    \frac{\partial^{\odot 3}}{3!}\odot (\mathbf{x};\frac{\mathbf{x}^{\odot 2}}{2!})\\
\end{array}
\right)$ the equality  $f^{\delta}_3\{\mathbf{y}\}=\det(h)^5\det(g)^{-10}f^{\partial}_3\{\mathbf{x}\}$ holds.

2. $27=6+21$, which implies that $k=6, j^6_1=2, j^6_2=5$ and for \[f^{\partial}_6\{\mathbf{x}\}= \det\left(
\begin{array}{c}
  \partial\odot (\frac{\mathbf{x}^{\odot 2}}{2!}, \frac{\mathbf{x}^{\odot 5}}{5!})\\
  \frac{\partial^{\odot 2}}{2!}\odot (\frac{\mathbf{x}^{\odot 2}}{2!}, \frac{\mathbf{x}^{\odot 5}}{5!})\\
    \frac{\partial^{\odot 3}}{3!}\odot (\frac{\mathbf{x}^{\odot 2}}{2!}, \frac{\mathbf{x}^{\odot 5}}{5!})\\
    \frac{\partial^{\odot 4}}{4!}\odot (\frac{\mathbf{x}^{\odot 2}}{2!}, \frac{\mathbf{x}^{\odot 5}}{5!})\\
    \frac{\partial^{\odot 5}}{5!}\odot (\frac{\mathbf{x}^{\odot 2}}{2!}, \frac{\mathbf{x}^{\odot 5}}{5!})\\
    \frac{\partial^{\odot 6}}{6!}\odot (\frac{\mathbf{x}^{\odot 2}}{2!}, \frac{\mathbf{x}^{\odot 5}}{5!})\\
\end{array}
\right)\] the equality  $f^{\delta}_6\{\mathbf{y}\}=\det(h)^{39}\det(g)^{-56}f^{\partial}_6\{\mathbf{x}\}$ holds.

3. $44=6+10+28$, which implies that $k=8, j^8_1=1, j^8_2=3, j^8_3=6$ and for $f^{\partial}_8\{\mathbf{x}\}$ the equality  $f^{\delta}_8\{\mathbf{y}\}=\det(h)^{67}\det(g)^{-120}f^{\partial}_8\{\mathbf{x}\}$ holds.

So in this case one has column vectors \[M_1^{\partial}\langle \mathbf{x}\rangle=\frac{\partial\odot f_6^{\partial}\langle \mathbf{x}\rangle}{-56f_6^{\partial}\langle \mathbf{x}\rangle}-\frac{\partial\odot f_3^{\partial}\langle \mathbf{x}\rangle}{-10f_3^{\partial}\langle \mathbf{x}\rangle},\ \ \
M_2^{\partial}\langle \mathbf{x}\rangle=\frac{\partial\odot f_8^{\partial}\langle \mathbf{x}\rangle}{-120f_8^{\partial}\langle \mathbf{x}\rangle}-\frac{\partial\odot f_3^{\partial}\langle \mathbf{x}\rangle}{-10f_3^{\partial}\langle \mathbf{x}\rangle},\] and for the second order square matrix $M^{\partial}\langle \mathbf{x}\rangle$ consisting of these columns
equality $(3)$ also holds.

 B. For $n=4$ one has the sequence \begin{eqnarray}\left(\begin{array}{c}
  n-1+i \\
  n-1 \\
\end{array}\right)_{i=1,2,3,...}=\{4, 10, 20, 35, 56, ...\}.\end{eqnarray}

Representing the elements of sequence $(4)$ as the sums of different elements of this sequence, when it is possible, one has:

1. $10=10$, which implies that $k=3, j^3_1=2$ and for $f^{\partial}_3\{\mathbf{x}\}= \det\left(
\begin{array}{c}
\frac{\mathbf{x}^{\odot 2}}{2!}\\
  \partial\odot \frac{\mathbf{x}^{\odot 2}}{2!}\\
  \frac{\partial^{\odot 2}}{2!}\odot\frac{\mathbf{x}^{\odot 2}}{2} \\
    \frac{\partial^\odot 3}{3!}\odot\frac{\mathbf{x}^{\odot 2}}{2}\\
   \end{array}
\right)$ the equality  $f^{\delta}_3\{\mathbf{y}\}=\det(h)^5\det(g)^{-10}f^{\partial}_3\{\mathbf{x}\}$ holds.

2. $45=10+35$, which implies that $k=8, j^8_1=2, j^8_2=4 $ and for \[f^{\partial}_8\{\mathbf{x}\}=\det\left(
\begin{array}{c}
(\frac{\mathbf{x}^{\odot 2}}{2},\frac{\mathbf{x}^{\odot 4}}{4!})\\
  \partial\odot (\frac{\mathbf{x}^{\odot 2}}{2},\frac{\mathbf{x}^{\odot 4}}{4!})\\
  \frac{\partial^{\odot 2}}{2!}\odot (\frac{\mathbf{x}^{\odot 2}}{2},\frac{\mathbf{x}^{\odot 4}}{4!})\\
    \frac{\partial^{\odot 3}}{3!}\odot (\frac{\mathbf{x}^{\odot 2}}{2},\frac{\mathbf{x}^{\odot 4}}{4!})\\
    \frac{\partial^{\odot 4}}{4!}\odot (\frac{\mathbf{x}^{\odot 2}}{2},\frac{\mathbf{x}^{\odot 4}}{4!})\\
    \frac{\partial^{\odot 5}}{5!}\odot (\frac{\mathbf{x}^{\odot 2}}{2},\frac{\mathbf{x}^{\odot 4}}{4!})\\
     \frac{\partial^{\odot 6}}{6!}\odot (\frac{\mathbf{x}^{\odot 2}}{2},\frac{\mathbf{x}^{\odot 4}}{4!})\\
      \frac{\partial^{\odot 7}}{7!}\odot (\frac{\mathbf{x}^{\odot 2}}{2},\frac{\mathbf{x}^{\odot 4}}{4!})\\
       \frac{\partial^{\odot 8}}{8!}\odot (\frac{\mathbf{x}^{\odot 2}}{2},\frac{\mathbf{x}^{\odot 4}}{4!})\\
    \end{array}
\right)\] the equality  $f^{\delta}_8\{\mathbf{y}\}=\det(h)^{40}\det(g)^{-120}f^{\partial}_8\{\mathbf{x}\}$ holds.

3. $55=20+35$, which implies that $k=9, j^9_1=3, j^9_2=4 $ and for $f^{\partial}_9\{\mathbf{x}\}$ the equality  $f^{\delta}_9\{\mathbf{y}\}=\det(h)^{50}\det(g)^{-165}f^{\partial}_9\{\mathbf{x}\}$ holds.

So for the second order square matrix $M^{\partial}\langle \mathbf{x}\rangle$ consisting of the columns $M_1^{\partial}\langle \mathbf{x}\rangle, M_2^{\partial}\langle\mathbf{x}\rangle$
equality $(3)$ is valid, where \[M_1^{\partial}\langle \mathbf{x}\rangle=\frac{\partial\odot f_8^{\partial}\langle \mathbf{x}\rangle}{-120f_8^{\partial}\langle \mathbf{x}\rangle}-\frac{\partial\odot f_3^{\partial}\langle \mathbf{x}\rangle}{-10f_3^{\partial}\langle \mathbf{x}\rangle},\ \ \
M_2^{\partial}\langle \mathbf{x}\rangle=\frac{\partial\odot f_9^{\partial}\langle \mathbf{x}\rangle}{-165f_9^{\partial}\langle \mathbf{x}\rangle}-\frac{\partial\odot f_3^{\partial}\langle \mathbf{x}\rangle}{-10f_3^{\partial}\langle \mathbf{x}\rangle}.\]

B'. If one uses for this purpose the second matrix representation then representing the  elements of sequence $(6)$ as the sums of different elements of sequence $(7)$, when it is possible, one has:

1. $14=4 +10$, which implies that $k=4, j^4_1=1, j^4_2=2$ and for $f^{\partial}_4\{\mathbf{x}\}= \det\left(
\begin{array}{c}
  \partial\odot (\mathbf{x}, \frac{\mathbf{x}^{\odot 2}}{2!})\\
  \frac{\partial^{\odot 2}}{2!}\odot (\mathbf{x}, \frac{\mathbf{x}^{\odot 2}}{2!}) \\
     \end{array}
\right)$ the equality  $f^{\delta}_4\{\mathbf{y}\}=\det(h)^6\det(g)^{-20}f^{\partial}_4\{\mathbf{x}\}$ holds.

2. $20=20$, which implies that $k=5, j^5_1=3 $ and for \[f^{\partial}_8\{\mathbf{x}\}=\det\left(
\begin{array}{c}
  \partial\odot \frac{\mathbf{x}^{\odot 3}}{3}\\
  \frac{\partial^{\odot 2}}{2!}\odot \frac{\mathbf{x}^{\odot 3}}{3}\\
    \frac{\partial^{\odot 3}}{3!}\odot \frac{\mathbf{x}^{\odot 3}}{3}\\
    \frac{\partial^{\odot 4}}{4!}\odot \frac{\mathbf{x}^{\odot 3}}{3}\\
    \frac{\partial^{\odot 5}}{5!}\odot \frac{\mathbf{x}^{\odot 3}}{3}\\
     \end{array}
\right)\] the equality  $f^{\delta}_5\{\mathbf{y}\}=\det(h)^{15}\det(g)^{-35}f^{\partial}_5\{\mathbf{x}\}$ holds.

3. $35=35$, which implies that $k=7, j^7_1=4$ and for $f^{\partial}_7\{\mathbf{x}\}$ the equality  $f^{\delta}_7\{\mathbf{y}\}=\det(h)^{35}\det(g)^{-84}f^{\partial}_7\{\mathbf{x}\}$ holds.

So for the second order square matrix $M^{\partial}\langle \mathbf{x}\rangle$ consisting of the columns $M_1^{\partial}\langle \mathbf{x}\rangle, M_2^{\partial}\langle\mathbf{x}\rangle$
equality $(3)$ is valid, where \[M_1^{\partial}\langle \mathbf{x}\rangle=\frac{\partial\odot f_5^{\partial}\langle \mathbf{x}\rangle}{-35f_5^{\partial}\langle \mathbf{x}\rangle}-\frac{\partial\odot f_3^{\partial}\langle \mathbf{x}\rangle}{-20f_3^{\partial}\langle \mathbf{x}\rangle},\ \ \
M_2^{\partial}\langle \mathbf{x}\rangle=\frac{\partial\odot f_7^{\partial}\langle \mathbf{x}\rangle}{-84f_7^{\partial}\langle \mathbf{x}\rangle}-\frac{\partial\odot f_3^{\partial}\langle \mathbf{x}\rangle}{-20f_3^{\partial}\langle \mathbf{x}\rangle}.\]
\end{ex}

Due to the construction we would like to formulate the following combinatorial question.

\begin{que} Let $m, n$ be any fixed natural numbers. Is it true that infinitely many elements of the sequence $\left(\begin{array}{c}
  m+k \\
  m \\
\end{array}\right)_{k=1,2,3,...}$ ( as well as of the sequence $(\left(\begin{array}{c}
  m+k \\
  m \\
\end{array}\right)-1)_{k=1,2,3,...})$ are representable as the sums of different elements of the sequence $\left(\begin{array}{c}
  n+k \\
  n \\
\end{array}\right)_{k=1,2,3,...}$ \emph{?}\end{que}


\begin{thebibliography}{99}

\bibitem{Ol1} Olver P.J., and Pohjanpelto J., Moving frames for Lie
pseudo-groups, \emph{Canadian J. Math.}, 60, 2008, pp.1336-1386.

\bibitem{Ol2} Olver P.J. Differential Invariants of Surfaces, \emph{Diff.Geom.Appl.}, 27 (2009), pp.230-239.

\bibitem{Ol3}  Olver P.J., Moving frames and differential invariants in centro-affine geometry,
  \emph{Lobachevsky J. Math.}, 31, 2010, pp.77-89.

 \bibitem{B1} Bekbaev U., On differential rational invariants of patches with respect to motion groups. J. Phys.: Conf. Ser.(to appear)

 \bibitem{B2} Bekbaev U., On the field of differential rational invariants
 of a subgroup \\ of affine group (Partial differential case),
 \emph{http://front.math.ucdavis.edu/, arXiv:math.AG/0609252}, 12 p.

 \bibitem{B3} Bekbaev U., On relation between algebraic and ordinary differential algebraic invariants of motion groups,
 \emph{Lobachevskii J. Math.}, 2014, 35(3), pp.172-184,
 DOI:10.1134/S1995080214030044.

 \bibitem{B4} Bekbaev U., Rakhimov I.S., Muminov K., On differential rational invariants of classical motion groups,
  \emph{Journal of Algebra and its Applications}, 14(2), 2015,
 DOI: 10.1142/S0219498815500231.

\bibitem{We} Weyl H.,  \emph{The Classical Groups, Their Invariants and
Representations}, Princeton Academic Press, 1997, 320 p.

\bibitem{Kr} Kraft H., Procesi C., \emph{Classical invariant theory. A primer,}
1996, 123 p.

\bibitem{Kh} Khadjiev Dj., \emph{Application of Invariant  Theory to Differential
Geometry of curves}, FAN, Tashkent, 1988 (Russian).

\bibitem{B5} Bekbaev U.D., An algebraic approach to invariants of surfaces,
\emph{Proceedings of IPTA Research \& Development Exposition 2003. 9-12
October 2003. Putra World Trade Center, Kuala Lumpur. Vol. 5:
Science and Engineering},
University Putra Malaysia Press, Serdang, Selangor, Malaysia, pp.299-306.

\bibitem{B6} Bekbaev U., Differential rational invariants of Hypersurfaces
relative to Affine Group, \emph{Prociding Pengintegrasion Teknologi
dalam Sains Matematik}, Pusat
Pengajian Sains Matemarik, USM, 1999, pp.58-64.

\bibitem{B7} Bekbaev U., On differential rational invariants of finite
subgroups of Affine group. \emph{Bull. of Malaysian
Math. Soc.}, 2005, 28(1), pp.55-60.

\bibitem{Ko}  Kolchin E.R.,  \emph{Differential Algebra and Algebraic Groups},
Academic Press, New York, 1973, 448 p.

\bibitem{Ka} Kaplansky I., \emph{An Introduction to Differential algebra}, Paris,
Hermann, 1957.

\bibitem{B8} Bekbaev U., Matrix Representations for Symmetric and Antisymmetric maps,\\
\emph{http://front.math.ucdavis.edu/, arXiv:1010.2579 math.RA.}.

\bibitem{B9} Bekbaev U., A formula for higher order partial derivatives,
\emph{International conference on Operator algebras and related topics}, Tashkent, National University of Uzbekistan, 12-14 September, 2012.

 \bibitem{W} Weinstein A., Groupoids: Unifying Internal and External
Symmetry, \emph{Notices of the AMS}, 43(7), pp.744-752.

  \end{thebibliography}
\end{document}